\documentclass[11pt]{amsart}
\usepackage{amsfonts}
\usepackage{amsmath,graphicx,amssymb,amsthm}
\usepackage{amsmath}
\usepackage{amssymb}
\usepackage{graphicx}
\usepackage{comment}
\usepackage[backref]{hyperref}
\providecommand{\U}[1]{\protect\rule{.1in}{.1in}}

\usepackage[all]{xypic} 
\usepackage{tikz-cd}

\def\U{\mathcal U}

\def\amslatex{$\mathcal{A}\kern-.1667em\lower.5ex\hbox{$M$}\kern-.125em\mathcal{S}$-\LaTeX}

\newcommand{\abs}[1]{\left\vert#1\right\vert}
\newcommand{\norm}[1]{\left\Vert#1\right\Vert}
\newcommand{\h}{\mathcal{H}}
\newtheorem{set}{set}[section]

\newcommand{\enifed}{\mathrel{\hbox{$\equiv$\hskip -.90em \lower .47ex \hbox{$\rightharpoondown$}}}}

\newtheorem{theorem}[set]{Theorem}
\theoremstyle{plain}

\newtheorem{remark}[set]{Remark}

\allowdisplaybreaks[4]

\begin{document}
\title[The Modular Symmetry of Markov Maps]{The Modular Symmetry of Markov Maps}
\author{Jon P. Bannon}
\address{Siena College Department of Mathematics, 515 Loudon Road, Loudonville,
NY\ 12211, USA}
\author{Jan Cameron}
\address{Department of Mathematics and Statistics, Vassar College, Poughkeepsie, NY 12604, USA}
\author{Kunal Mukherjee}
\address{Indian Institute of Technology Madras, Chennai 600 036, India}

\begin{abstract}
A state--preserving automorphism of a von Neumann algebra induces a canonical unitary operator on the GNS Hilbert space of the state which fixes the vacuum. This unitary commutes with both the modular operator of the state and its modular conjugation. We prove an extension of this result for state--preserving unital completely positive maps.
\end{abstract}
\maketitle

%%%%%%%%%%%%%%%%%%%%%%%%%%%%%%%%%%%%%%%%%%%%%%%%%%%%%%%%%%%%%%%%%%%%%%%%%%%%%%%%%%%%%%%%%%%%%%%%%%%%%%%%%%%%%%%%%%%%%%%%%%%%%%%%%

%%%%%%%%%%%%%%%%%%%%%%%%%%%%%%%%%%%%%%%%%%%%%%
\section{Introduction}
The starting point of this note is the classic paper of Haagerup \cite{Ha} on standard forms, in which it was observed that any state-preserving automorphism of a von Neumann algebra can be extended to a unitary operator on the GNS Hilbert space of the state.  It was proved there, by an elegant argument with the polar decomposition, that this unitary extension commutes with both the modular operator of the state and its associated modular conjugation. It is natural to ask whether these results can be generalized to the setting of state-preserving completely positive maps between von Neumann algebras; the main result of this paper answers this question affirmatively for an important class of such maps.

If $M$ and $N$ are von Neumann algebras with normal, faithful states $\varphi:M \rightarrow \mathbb{C}$ and $\rho:N \rightarrow \mathbb{C}$, a linear map $\Phi: M \rightarrow N$ is called a $(\varphi,\rho)$--Markov map if $\Phi$ is unital, completely positive, and satisfies 
\[\rho \circ \Phi = \varphi, \quad \text{ and } \quad \sigma_t^\rho \circ \Phi = \Phi \circ \sigma_t^\varphi, \text{ for all }  t \in \mathbb{R}.\]
Any such map is automatically normal, and extends in a natural way to a Hilbert space operator $T_\Phi: L^2(M,\varphi) \rightarrow L^2(N,\rho),$ which we call the $L^2$-extension of $\Phi$.  The main result of this note is that the $L^2$-extension $T_\Phi$ of any $(\varphi,\rho)$-Markov map $\Phi$ between von Neumann algebras $(M,\varphi)$ and $(N,\rho)$ satisfies an analogous property to the state-preserving automorphisms studied in \cite{Ha}.  In particular, we show that for any such $\Phi$, the map $T_\Phi$ intertwines the anti-linear isometries associated to $\varphi$ and $\rho$, and -- up to passing to the closure of a (necessarily) densely-defined operator -- also intertwines the modular operators of $\varphi$ and $\rho$.  

Markov maps between von Neumann algebras that connect a pair of states and intertwine their modular automorphism groups in this way have appeared recently in the von Neumann algebra literature in various contexts.  Anantharaman-Delaroche \cite{AD} used $(\varphi,\rho)$--Markov maps in proving a noncommutative version of an ergodic theorem of Nevo and Stein, and asked whether all such maps were ``factorizable''. Haagerup and Musat later answered this question in the negative, on the way to their solution to the Asymptotic Quantum Birkhoff Conjecture \cite{HaMu}. It has also been shown (cf. Theorem 4.1 of \cite{CaSk}) that a version of the Haagerup Approximation Property for a von Neumann algebra $M$ with a fixed normal, faithful state $\varphi$ can be formulated in terms of $(\varphi,\varphi)$-Markov maps on $M$.

The initial motivation for studying this problem arose in an attempt to develop a general theory of joinings of $W^*$-dynamical systems (the initial steps of which will appear in the forthcoming paper \cite{BCM}), i.e., $(\varphi,\rho)$--Markov maps as above that also intertwine actions of a group $G$ on the associated GNS Hilbert spaces. Many properties of classical measurable dynamical systems are defined relative to a given subsystem  -- for example, relative ergodicity and relative weak mixing, and relative independence of a pair of systems over a common subsystem -- but formulation of their noncommutative analogues poses technical challenges.   One issue that arises is the requirement of ``modular symmetry" of the canonical $L^{2}$--extension of a certain Markov map associated to the dynamical system.  Indeed, the result mentioned above (Proposition 3.7 of \cite{Ha}) on commutation of the $L^2$-extension of a state-preserving automorphism with both the modular operator and modular conjugation plays a role in previous work on joinings of  $W^{*}$--dynamical systems (see, for instance, Construction 3.4 of \cite{Du} and Lemma 2.4 of \cite{Du2}).  With a view toward further developments on joinings of $W^*$-dynamical systems, we therefore aim to establish a modular symmetry result in the more general setting of the induced operator associated to a Markov map.  Our main result is the following. %we provide the analogous result for general (time--covariant) quantum channels beyond state-preserving automorphisms, which is not at all immediate, but is essential for the study of relative noncommutative joinings. We prove:

\begin{theorem}\label{Commutation}
Let $(N,\rho)$ and $(M,\varphi)$ be two $W^{\ast}$--probability spaces, where $N,M$ are von Neumann algebras with separable 
preduals and $\rho$, $\varphi$ are faithful normal states on $N$ and $M$, respectively. Let $\Phi:N\rightarrow M$ be a unital completely positive $($u.c.p. in the sequel$)$ map such that $\varphi\circ\Phi=\rho$ and $\Phi \circ\sigma_{t}^{\rho}= \sigma_{t}^{\varphi} \circ \Phi$ for all $t\in \mathbb{R}$, where $\sigma_{t}^{\rho}, \sigma_{t}^{\varphi}$ denote the associated modular automorphisms. Denote by $\Delta_{\rho}$ and $\Delta_{\varphi}$ the associated modular operators and by $J_{\rho}$ and $J_{\varphi}$ the associated anti--linear isometries. Then, 
\begin{align*}
&(i) \text{ }\overline{\Delta_{\varphi}^{-s}T_{\Phi}\Delta_{\rho}^{s}}=T_{\Phi}, \text{ for all }s\in \mathbb{R},\\
&(ii) \text{ }J_{\varphi} T_{\Phi} J_{\rho}=T_{\Phi},
%&(iii) \text{ }\overline{S_{\varphi}T_{\Phi}S_{\rho}}=T_{\Phi},
\end{align*}
where $T_{\Phi}: \mathbf{B}(L^{2}(N,\rho))\rightarrow \mathbf{B}(L^{2}(M,\varphi))$ is the $L^{2}$--extension of $\Phi$.
\end{theorem}

\vspace{5mm}
\textbf{Acknowledgements:}
Work on this paper was initiated during a visit of KM to Vassar College and Siena College in 2012, partially supported by Vassar's Rogol Distinguished Visitor program.  We thank the Rogol Fund for this support.  JC's research was partially supported by a research travel grant from the Simons Foundation, and by Simons Foundation Collaboration Grant for Mathematicians \#319001. JB thanks Liming Ge and the Academy of Mathematics and Systems Science of the Chinese Academy of Sciences for their hospitality and support. KM thanks Serban \c{S}tr\v{a}til\v{a} for helpful conversations. The authors thank Ken Dykema, David Kerr and Vern Paulsen, for their helpful feedback on this project.

%%%%%%%%%%%%%%%%%%%%%%%%%%%%%%%%%%%%%%%%%%%%%%%%%%%%%%%%%%%%%%%%%%%%%%%%%%

\section{Preliminaries}

All von Neumann algebras in this paper have separable preduals. Let $M$ be a von Neumann algebra with a faithful, normal state $\varphi$. Denote by $L^{2}(M,\varphi)$ and $\Omega_{\varphi}$ the associated GNS Hilbert space and its canonical unit cyclic and separating vector, and let $M$ act on $L^{2}(M,\varphi)$ via left multiplication. We will denote the inner product and norm on $L^{2}(M,\varphi)$ by $\langle\cdot,\cdot\rangle_{\varphi}$ and $\norm{\cdot}_{\varphi}$, respectively.

We recall, without proof, the following facts that are needed in the sequel. The conjugate--linear map defined by $S_{0,\varphi}x\Omega_{\varphi}=x^{\ast}\Omega_{\varphi}$ for all $x\in M$ is closable with closure $S_{\varphi}$ having polar decomposition $J_{\varphi}\Delta_{\varphi}^{1/2}$.  In fact, the adjoint $S_{\varphi}^{\ast}$ is the closure of the closable linear map on $L^{2}(M,\varphi)$ defined by $F_{0}x^{\prime}\Omega_{\varphi}=(x^{\prime})^{\ast}\Omega_{\varphi}$ for all $x^{\prime}\in M^{\prime}$,
and the polar decomposition of $S_{\varphi}^{\ast}$ is $J_{\varphi}%
\Delta_{\varphi}^{-1/2}$. The conjugate--linear map $J_{\varphi}:L^{2}(M,\varphi) \rightarrow L^{2}(M,\varphi)$ satisfies $J_{\varphi}^{2}=1$ and $\langle J_{\varphi}\xi,J_{\varphi}\eta\rangle_{\varphi}=\langle\eta,\xi\rangle_{\varphi}$ for all $\xi,\eta\in L^{2}(M,\varphi)$, i.e. $J_{\varphi}=J_{\varphi}^{\ast}$ as a conjugate--linear map. Tomita's modular operator is the positive, self--adjoint operator $\Delta_{\varphi}=S_{\varphi}^{\ast}S_{\varphi}$. The operator $\Delta_{\varphi}$ is invertible and satisfies
$J_{\varphi}\Delta_{\varphi}J_{\varphi}=\Delta_{\varphi}^{-1}$, as well as $S_{\varphi}=J_{\varphi}\Delta_{\varphi
}^{1/2}=\Delta_{\varphi}^{-1/2}J_{\varphi}$ and $S_{\varphi}^{\ast}%
=J_{\varphi}\Delta_{\varphi}^{-1/2}=\Delta_{\varphi}^{1/2}J_{\varphi}$. Furthermore, $\Delta_{\varphi}^{it}J_{\varphi}=J_{\varphi}\Delta_{\varphi}^{it}$ and $\Delta_{\varphi}^{it}\Omega_{\varphi}=J_{\varphi}\Omega_{\varphi}=\Omega_{\varphi}$ for all $t\in\mathbb{R}$. By the fundamental theorem of Tomita and Takesaki, $\Delta_{\varphi}^{it}M\Delta_{\varphi}^{-it}=M$
for all $t\mathbb{\in R}$, and $J_{\varphi}MJ_{\varphi}=M^{\prime}$. Recall that $\sigma_{t}^{\varphi}(x)=\Delta_{\varphi}^{it}x\Delta_{\varphi}^{-it}$ for all $x\in M$ defines the modular automorphism of $M$ associated to $t \in \mathbb{R}$. For more detail we refer the reader to \cite{St} and \cite{Ta}.

Let $H$ be a densely--defined positive self--adjoint nonsingular operator on a Hilbert space $\h$. Then $H$ and $1$ generate an abelian von Neumann algebra $\mathcal{A}\subset \textbf{B}(\h)$ and $H$ is affiliated to $\mathcal{A}$. If $f$ and $g$ are complex--valued Borel measurable functions on $\mathbb{C}$ such that $f=g$ on 
 $\sigma(H)\setminus\{0\}$ $($or equivalently $f=g$ on $(0,\infty))$, then $f(H)=g(H)$ and these (possibly densely defined) operators are closed. For $z\in \mathbb{C}$, let 
\begin{equation*}
f_{z}(u)=\begin{cases} e^{zLog \text{ }u}, \text{ }u\in \mathbb{C}_{s}=\{w\in\mathbb{C}: w\neq -\abs{w}\}, \\
0, \text{ otherwise},\end{cases}
\end{equation*} 
where $Log$ is the principal branch of the logarithm on $\mathbb{C}_{s}$. Then $f_{z_{1}}f_{z_{2}}= f_{z_{1}+z_{2}}$ for all $z_{1},z_{2}\in\mathbb{C}$. Thus, writing $H^{z}=f_{z}(H)$ and using the functional calculus for unbounded operators $($Theorem 5.6.26 \cite{KRI}$)$, one has $\overline{H^{z_{1}}H^{z_{2}}} =\overline{f_{z_{1}}(H)f_{z_{2}}(H)}= (f_{z_{1}}f_{z_{2}})(H)= f_{z_{1}+z_{2}}(H)$. Let $\log$ be the Borel function on $\mathbb{C}$ defined as $Log\text{ }u$ when $u\in\mathbb{C}_{s}$ and $0$ when $u=-\abs{u}$, it follows that $\mathbb{R}\ni t\mapsto e^{it\log H}$ is a strong--operator continuous one--parameter group of unitaries on $\h$. Note that the functions $s\mapsto e^{it\log s}$ and $f_{it}$ agree on $\sigma(H)\setminus \{0\}\subseteq \mathbb{C}_{s}$ and thus 
\begin{align*}
H^{it} = f_{it}(H) = e^{it\log H}.
\end{align*}

Since $H$ is closed and injective, its inverse operator $H^{\prime}$ is closed and densely defined. Let $H^{-1}$ denote $f_{-1}(H)$. Then $\overline{HH^{-1}}=1$. Note that $H^{\prime}$ and $H^{-1}$ are both affiliated to $\mathcal{A}$ and both are inverses to $H$ in the algebra of all closed operators affiliated to $\mathcal{A}$. Thus,
\begin{align*}
H^{-1}=\overline{(\overline{H^{\prime}H)}H^{-1}}=\overline{H^{\prime}\overline{(HH^{-1})}}=H^{\prime}.
\end{align*}
 
Again $H^{-1}$ is positive, self--adjoint and nonsingular and $f_{z}\circ f_{-1}=f_{-z}$ for all 
$z\in\mathbb{C}$. Thus, $(H^{-1})^{z}=H^{-z}$ for all $z\in\mathbb{C}$. While $f_{-1}\circ f_{z}$ need not agree with $f_{-z}$ for all $z\in\mathbb{C}$, these do agree when $z$ is a real number. Thus for any $t\in \mathbb{R}$, 
\begin{align*}
(H^{-1})^{t}=H^{-t}=(H^{t})^{-1},
\end{align*}
so that $(H^{-1})^{t}, (H^{t})^{-1}$ and $H^{-t}$ are all inverses of the operator $H^{t}$. 

Note that $H^{z}$ is affiliated to $\mathcal{A}$ for all $z\in \mathbb{C}$. Recall that if $T$ is closed and $A$ is bounded and everywhere defined on $\h$, then $TA$ is closed. Thus, if $f$ and $g$ are Borel functions whose domains each contain $\sigma(H)$ and $g$ is bounded, then  $f(H)g(H)$ is densely defined and closed. Thus $(fg)(H)=\overline{f(H)g(H)}=f(H)g(H)$  and for all $t\in\mathbb{R}$ and $z \in \mathbb{C}$ 
\begin{align*}
&H^{z}H^{it}=\overline{H^{z}H^{it}}=H^{z+it}.
\end{align*}
As $H^{z}$ is affiliated to $\mathcal{A}$ and $H^{it}$ is a unitary in this abelian algebra,  
\begin{align*}
&H^{it}H^{z}=H^{z}H^{it}=H^{z+it}
\end{align*}
for all $t\in\mathbb{R}$ and $z \in \mathbb{C}$. 

Let $\mathcal{G}(\cdot)$ denote the graph of an operator, and let $0\neq \alpha\in \mathbb{R}$. Then it is a standard fact of the Tomita--Takesaki theory that $(\xi,\eta)\in \mathcal{G}(H^{\alpha})$ if and only if the function $i\mathbb{R}\ni it \mapsto H^{it}\xi\in \h$ defined on the imaginary axis has a continuous extension $F$ to the strip $\mathcal{D}_{\alpha}$ which is analytic in its interior $\mathcal{D}_{\alpha}^{\circ}$ and $F(\alpha)=\eta$, where $\mathcal{D}_{\alpha}=\{z\in \mathbb{C}: 0\leq \Re{z}\leq \alpha\}$ if $\alpha> 0$ and $\mathcal{D}_{\alpha}=\{z\in \mathbb{C}: \alpha\leq \Re{z}\leq 0\}$ if $\alpha < 0$. By convention $H^{0}=1$ and thus $\mathcal{D}_{0}=\mathfrak{D}(H^{0})=\h$ $($c.f. Lemma $\rm{VI}.2.3$, \cite{Ta}$)$. We will use the above facts with $H$ replaced by the modular operators $\Delta_{\rho}$ and $\Delta_{\varphi}$ in the statement of Theorem \ref{Commutation}. In what follows, if $H=\Delta_{\omega}$ with $\omega\in\{\rho,\varphi\}$, we will denote %the domain $\mathcal{D}_{\alpha}$ by $\mathcal{D}_{\omega,\alpha}$ and 
$\mathcal{A}$ by $\mathcal{A}_{\omega}$.

\section{Main Results}

As above, let $N$ and $M$ be von Neumann algebras equipped with faithful, normal states
$\rho$ and $\varphi$, respectively. Let $\Phi:N\rightarrow M$ be a u.c.p. map such that $\varphi\circ\Phi=\rho$. Then $\Phi$ is automatically normal by a classic result of Tomiyama
\cite{To}. Define $T_{\Phi}:L^{2}(N,\rho)\rightarrow L^{2}(M,\varphi)$ by $T_{\Phi}(x\Omega_{\rho}
)=\Phi(x)\Omega_{\varphi}$ for all $x\in N$. Thus, $T_{\Phi}$ is, a priori, an unbounded
operator. However, by Kadison's inequality we have
\begin{align*}
\langle T_{\Phi}(x\Omega_{\rho}),T_{\Phi}(x\Omega_{\rho})\rangle_{\varphi}  
&=\langle\Phi(x)\Omega_{\varphi},\Phi(x)\Omega_{\varphi}\rangle_{\varphi}
=\varphi(\Phi(x^{\ast})\Phi(x))\\
\nonumber&  \leq\varphi(\left\Vert \Phi(1)\right\Vert \Phi(x^{\ast}x))=\rho(x^{\ast}x)\\
\nonumber&  =\langle x\Omega_{\rho},x\Omega_{\rho}\rangle_{\rho}.
\end{align*}
Thus, $T_{\Phi}$ extends to a bounded operator from $L^{2}(N,\rho)$ to
$L^{2}(M,\varphi)$ of norm $1$, as $\norm{T_{\Phi}(\Omega_{\rho})}_{\varphi}=1$. 
Moreover, if $\Phi\circ\sigma_{t}^{\rho}=\sigma_{t}^{\varphi}\circ\Phi$ for all
$t\in\mathbb{R}$, then by a result of Accardi--Cecchini \cite{AC}
there exists a normal u.c.p. map $\Phi^{\ast}:M\rightarrow N$
satisfying
\begin{equation}
\rho(\Phi^{\ast}(y)x)=\varphi(y\Phi(x))\label{Eq: AccardiCecciniAdjoint}
\end{equation}
for all $y\in M$ and $x\in N$. It follows that $T_{\Phi}^{*} = T_{\Phi^{*}}$. Furthermore, for all $t\in \mathbb{R}$ 
\begin{align}\label{Eq: Crucial Equation}
&T_{\Phi}\Delta^{it}_{\rho}=\Delta^{it}_{\varphi}T_{\Phi}.
\end{align}
Let $CP(N,M,\rho,\varphi)$ be the set
\begin{align*}
\{\Phi:N\rightarrow M |\text{ } \Phi \text{ is u.c.p., } \varphi\circ\Phi=\rho \text{ and } \Phi\circ \sigma_{t}^{\rho}=\sigma_{t}^{\varphi}\circ \Phi \text{ }\forall t\in \mathbb{R}\}. 
\end{align*}
 
The next intermediary result is natural and encodes a lot of information, but its proof is tedious. We provide a detailed proof since we cannot find one in the literature. 

\begin{theorem}\label{Commute}
Let $N$ and $M$ be von Neumann algebras equipped with faithful normal states $\rho$ and $\varphi$ respectively. Let $\Phi\in CP(N,M,\rho,\varphi)$.  Then
\begin{align*}
\mathfrak{D}_{\rho}=\bigl\{\xi\in N\Omega_{\rho}: \exists \text{ }F:\mathbb{C}\rightarrow N\Omega_{\rho}, \text{ } F \text{ is entire and } F(it)=\Delta_{\rho}^{it}\xi \text{ }\forall t\in\mathbb{R}\bigr\}
\end{align*}
is a core for $\Delta_{\varphi}^{z}T_{\Phi}$ and 
\begin{align*}
&\overline{T_{\Phi}\Delta_{\rho}^{z}}= \Delta_{\varphi}^{z}T_{\Phi}, \text{ }z\in\mathbb{C}.
\end{align*}
\end{theorem}

We remark that the containment $T_{\Phi}\Delta_{\rho}^{z}\subseteq \Delta_{\varphi}^{z}T_{\Phi}$ is easy to verify but the equality as stated above requires argument. 

\begin{proof}
\noindent\textbf{Step 1}: In this step, we justify that it is enough to prove the assertions when $z$ is real. Let $z=s+it$. Then, as discussed above, $\Delta_{\omega}^{z}=\Delta_{\omega}^{s}\Delta_{\omega}^{it}=\Delta_{\omega}^{it}\Delta_{\omega}^{s}$, where $\omega\in \{\rho,\varphi\}$. Assume that $T_{\Phi}\Delta_{\rho}^{s}$ is closable, $\overline{T_{\Phi}\Delta_{\rho}^{s}}= \Delta_{\varphi}^{s}T_{\Phi}$ and $\mathfrak{D}_{\rho}$ is a common core of $\overline{T_{\Phi}\Delta_{\rho}^{s}}$ and $\Delta_{\varphi}^{s}T_{\Phi}$. 

Note that $\overline{T_{\Phi}\Delta_{\rho}^{s}}\Delta_{\rho}^{it}$ is closed. Also note that $\Delta_{\rho}^{it}\mathfrak{D}_{\rho}=\mathfrak{D}_{\rho}$. It follows that $T_{\Phi}\Delta_{\rho}^{s}\Delta_{\rho}^{it}$ is to be densely defined and so $\Delta_{\rho}^{-it}(T_{\Phi}\Delta_{\rho}^{s})^{*} \subseteq 
(T_{\Phi}\Delta_{\rho}^{s}\Delta_{\rho}^{it})^{*}$. Since $T_{\Phi}\Delta_{\rho}^{s}$ is closable, $(T_{\Phi}\Delta_{\rho}^{s})^{*}$ is densely defined, which forces 
$(T_{\Phi}\Delta_{\rho}^{s}\Delta_{\rho}^{it})^{*}$ to be densely defined, and consequently, 
$T_{\Phi}\Delta_{\rho}^{z}=T_{\Phi}\Delta_{\rho}^{s}\Delta_{\rho}^{it}$ is closable. 
Clearly $\overline{T_{\Phi}\Delta_{\rho}^{z}}\subseteq \overline{T_{\Phi}\Delta_{\rho}^{s}}\Delta_{\rho}^{it}$. 
For the other inclusion, let $(\xi, \overline{T_{\Phi}\Delta_{\rho}^{s}}\Delta_{\rho}^{it}\xi)\in \mathcal{G}(\overline{T_{\Phi}\Delta_{\rho}^{s}}\Delta_{\rho}^{it})$. Then $\Delta_{\rho}^{it}\xi\in\mathfrak{D}(\overline{T_{\Phi}\Delta_{\rho}^{s}})$, and since $\mathfrak{D}_{\rho}$ is a core for $\overline{T_{\Phi}\Delta_{\rho}^{s}}$ there exists $\xi_{n}\in \mathfrak{D}_{\rho}$ such that $\xi_{n}\rightarrow \Delta_{\rho}^{it}\xi$ and $\overline{T_{\Phi}\Delta_{\rho}^{s}}\xi_{n}\rightarrow \overline{T_{\Phi}\Delta_{\rho}^{s}}\Delta_{\rho}^{it}\xi$. Since $\Delta_{\rho}^{it}$ keeps $\mathfrak{D}_{\rho}$ invariant, $\mathfrak{D}(T_{\Phi}\Delta_{\rho}^{s}\Delta_{\rho}^{it})\supseteq\mathfrak{D}_{\rho}\ni\Delta_{\rho}^{-it}\xi_{n} \rightarrow \xi$ and  
$T_{\Phi}\Delta_{\rho}^{s}\Delta_{\rho}^{it}(\Delta_{\rho}^{-it}\xi_{n})\rightarrow \overline{T_{\Phi}\Delta_{\rho}^{s}}\Delta_{\rho}^{it}\xi$; thus $(\xi, \overline{T_{\Phi}\Delta_{\rho}^{s}}\Delta_{\rho}^{it}\xi)\in 
\mathcal{G}(\overline{T_{\Phi}\Delta_{\rho}^{s}\Delta_{\rho}^{it}})$. It follows that $\overline{T_{\Phi}\Delta_{\rho}^{s}}\Delta_{\rho}^{it}=\overline{T_{\Phi}\Delta_{\rho}^{s}\Delta_{\rho}^{it}}$. Consequently, 
by Eq. \eqref{Eq: Crucial Equation}, we have 
\begin{align*}
\Delta_{\varphi}^{z}T_{\Phi}=\Delta_{\varphi}^{s+it}T_{\Phi}=
\Delta_{\varphi}^{s}\Delta_{\varphi}^{it}T_{\Phi}=\Delta_{\varphi}^{s}T_{\Phi}\Delta_{\rho}^{it}
=\overline{T_{\Phi}\Delta_{\rho}^{s}}\Delta_{\rho}^{it}=\overline{T_{\Phi}\Delta_{\rho}^{s}\Delta_{\rho}^{it}}
=\overline{T_{\Phi}\Delta_{\rho}^{z}},
\end{align*}
and $\mathfrak{D}_{\rho}$ is a core for $\Delta_{\varphi}^{z}T_{\Phi}=\overline{T_{\Phi}\Delta_{\rho}^{z}}$. Thus, it is sufficient to prove the assertions when $z\in\mathbb{R}$. 

\noindent\textbf{Step 2}: Let $z=s\in\mathbb{R}$. We claim that $\overline{T_{\Phi}\Delta_{\rho}^{s}}\subseteq \Delta_{\varphi}^{s}T_{\Phi}$. 

Let $(\xi,T_{\Phi}\Delta_{\rho}^{s}\xi)\in \mathcal{G}(T_{\Phi}\Delta_{\rho}^{s})$. Then $\xi\in \mathfrak{D}(\Delta_{\rho}^{s})$. Let 
$F: \mathcal{D}_{s}\rightarrow L^{2}(N,\rho)$ be a continuous function such that $F$ is analytic in $\mathcal{D}_{s}^{\circ}$, $F(it)=\Delta_{\rho}^{it}\xi$ for all $t\in\mathbb{R}$ and $F(s)=\Delta_{\rho}^{s}\xi$. Then by the uniqueness of analytic continuation and Eq. \eqref{Eq: Crucial Equation} it follows that $T_{\Phi}\circ F: \mathcal{D}_{s}\rightarrow L^{2}(M,\varphi)$ is continuous, analytic in $\mathcal{D}_{s}^{\circ}$ and $(T_{\Phi}\circ F)(it)=\Delta_{\varphi}^{it} T_{\Phi}(\xi)$. Hence $\xi\in \mathfrak{D}(\Delta_{\varphi}^{s}T_{\Phi})$ and $\Delta_{\varphi}^{s}T_{\Phi}(\xi)=(T_{\Phi}\circ F)(s)=T_{\Phi}(\Delta_{\rho}^{s}\xi)$. This shows that $T_{\Phi}\Delta_{\rho}^{s}\subseteq \Delta_{\varphi}^{s}T_{\Phi}$. Consequently, as $\Delta_{\varphi}^{s}T_{\Phi}$ is closed it follows that $T_{\Phi}\Delta_{\rho}^{s}$ is closable and $\overline{T_{\Phi}\Delta_{\rho}^{s}}\subseteq \Delta_{\varphi}^{s}T_{\Phi}$. 

\noindent\textbf{Step 3}: We now proceed to find a common subspace $\mathfrak{D}_{\rho}$ on which  $\overline{T_{\Phi}\Delta_{\rho}^{s}}$ and $\Delta_{\varphi}^{s}T_{\Phi}$ both agree. By the previous discussion, $\mathbb{R}\ni t\mapsto \Delta_{\rho}^{it}$ $($resp. $\Delta_{\varphi}^{it})$, is a strongly continuous unitary group with infinitesimal self--adjoint generator $($Hamiltonian$)$ $\log \Delta_{\rho}$ $($resp. $\log\Delta_{\varphi})$. Note that $\Delta_{\rho}^{it}N\Omega_{\rho}=N\Omega_{\rho}$, for all $t\in \mathbb{R}$. Also, if $f\in L^{1}(\mathbb{R})$, then the operator $\Delta_{\rho}(f)=\int_{\mathbb{R}}f(s)\Delta_{\rho}^{is}ds$ keeps $N\Omega_{\rho}$ invariant. Indeed, for $x\in N$, the map 
\begin{align*}
L^{2}(N,\rho)\times L^{2}(N,\rho)\ni (\xi,\eta) \mapsto \int_{\mathbb{R}}f(t)\langle\sigma_{t}^{\rho}(x)\xi,\eta\rangle_{\rho}dt
\end{align*}
defines a bounded, sesquilinear form yielding a bounded operator $\int_{\mathbb{R}}f(t)\sigma_{t}^{\rho}(x)dt$ on $L^{2}(N,\rho)$. A simple calculation shows that $\int_{\mathbb{R}}f(t)\sigma_{t}^{\rho}(x)dt$ commutes with 
$J_{\rho}yJ_{\rho}$ for each $y\in N$, and thus by Tomita's theorem $\int_{\mathbb{R}}f(t)\sigma_{t}^{\rho}(x)dt\in N$. Consequently, $\Delta_{\rho}(f)(N\Omega_{\rho})\subseteq N\Omega_{\rho}$. It follows that $\mathfrak{D}_{\rho}$ as defined in the statement of the theorem
is a core for $\Delta_{\rho}^{z}$, for any $z\in \mathbb{C}$, and in particular is a core for
$\Delta_{\rho}^{s}$ $($\cite{Ta} pp. 121--122$)$. Working analogously with $M$, one finds that $\mathfrak{D}_{\varphi}$ $($which is defined similarly to $\mathfrak{D}_{\rho})$ is a common core for $\Delta_{\varphi}^{z}$ for all $z\in\mathbb{C}$, and in particular is a core for 
$\Delta_{\varphi}^{s}$. We claim that 
\begin{align}\label{AgreeonD_U}
T_{\Phi}\Delta_{\rho}^{s}=\Delta_{\varphi}^{s}T_{\Phi} \text{ on } \mathfrak{D}_{\rho}.
\end{align}
Indeed, if $\xi\in \mathfrak{D}_{\rho}$ and $F:\mathbb{C}\rightarrow N\Omega_{\rho}$ is  such that $F$ is entire and $F(it)=\Delta_{\rho}^{it}\xi$ 
for all $t\in\mathbb{R}$, then $T_{\Phi}\circ F:\mathbb{C}\rightarrow M\Omega_{\varphi}$ is entire and $(T_{\Phi}\circ F)(it)=\Delta_{\varphi}^{it}T_{\Phi}(\xi)$ for all $t\in \mathbb{R}$. Thus, $T_{\Phi}(\xi)\in \mathfrak{D}_{\varphi}\subseteq \mathfrak{D}(\Delta_{\varphi}^{s})$ and $T_{\Phi}\Delta_{\rho}^{s}\xi=(T_{\Phi}\circ F)(s)=\Delta_{\varphi}^{s}T_{\Phi}(\xi)$. This establishes Eq. \eqref{AgreeonD_U}.

\noindent\textbf{Step 4}: We now claim that $\mathfrak{D}_{\rho}$ is a core for $\Delta_{\varphi}^{s}T_{\Phi}$. To see this, first note that $\mathcal{G}(\Delta_{\varphi}^{s}{T_{\Phi}}|_{\mathfrak{D}_{\rho}})\subseteq \mathcal{G}(\Delta_{\varphi}^{s}T_{\Phi})$ and the latter is closed. Consider any $(\xi_0,\Delta_{\varphi}^{s}T_{\Phi}(\xi_0))\in \mathcal{G}(\Delta_{\varphi}^{s}T_{\Phi})$ such that $(\xi_0,\Delta_{\varphi}^{s}T_{\Phi}(\xi_0))\in \mathcal{G}({\Delta_{\varphi}^{s}T_{\Phi\vert}}_{\mathfrak{D}_{\rho}})^{\perp}$. Then for all $\zeta\in\mathfrak{D}_{\rho}$, one has 
\begin{align*}
\langle \xi_{0}, \zeta\rangle_{\rho}&=-\langle \Delta_{\varphi}^{s}T_{\Phi}(\xi_{0}), \Delta_{\varphi}^{s}T_{\Phi}(\zeta)\rangle_{\varphi}\\
&=-\langle \Delta_{\varphi}^{s}T_{\Phi}(\xi_{0}), T_{\Phi}\Delta_{\rho}^{s}\zeta\rangle_{\varphi}\text{ (by Eq. \eqref{AgreeonD_U})}\\
&=-\langle T_{\Phi}^{*}\Delta_{\varphi}^{s}T_{\Phi}(\xi_{0}),\Delta_{\rho}^{s}\zeta\rangle_{\rho}.
\end{align*}
Since $\mathfrak{D}_{\rho}$ is a core of $\Delta_{\rho}^{s}$,
\begin{align}\label{HahnBanachSeperation}
\langle \xi_{0}, \zeta\rangle_{\rho}&=-\langle T_{\Phi}^{*}\Delta_{\varphi}^{s}T_{\Phi}(\xi_{0}),\Delta_{\rho}^{s}\zeta\rangle_{\rho}, \text{ for all }\zeta\in \mathfrak{D}(\Delta_{\rho}^{s}).
\end{align} 
Hence $T_{\Phi}^{*}\Delta_{\varphi}^{s}T_{\Phi}(\xi_{0})\in \mathfrak{D}(\Delta_{\rho}^{s})$, since $\Delta_{\rho}^{s}$ is self--adjoint and positive. Therefore, 
\begin{align*}
\langle \xi_{0}, \zeta\rangle_{\rho}&=-\langle \Delta_{\rho}^{s}T_{\Phi}^{*}\Delta_{\varphi}^{s}T_{\Phi}(\xi_{0}),\zeta\rangle_{\rho}, \text{ for all }\zeta\in \mathfrak{D}(\Delta_{\rho}^{s}).
\end{align*}

Since $\mathfrak{D}(\Delta_{\rho}^{s})$ is dense, so $\Delta_{\rho}^{s}T_{\Phi}^{*}\Delta_{\varphi}^{s}T_{\Phi}(\xi_{0})=-\xi_{0}$. Thus, $(T_{\Phi}^{*}\Delta_{\varphi}^{s}T_{\Phi}(\xi_{0}),-\xi_{0})\in \mathcal{G}(\Delta_{\rho}^{s})$. 
As $\Delta_{\rho}^{s},\Delta_{\varphi}^{s}$ are self--adjoint and positive it follows that, 
\begin{align}\label{Positivity}
0&\leq \langle \Delta_{\rho}^{s}T_{\Phi}^{*}\Delta_{\varphi}^{s}T_{\Phi}(\xi_{0}), T_{\Phi}^{*}\Delta_{\varphi}^{s}T_{\Phi}(\xi_{0})\rangle_{\rho}\\
\nonumber&= -\langle\xi_{0}, T_{\Phi}^{*}\Delta_{\varphi}^{s}T_{\Phi}(\xi_{0}) \rangle_{\rho}\\
\nonumber&= -\langle T_{\Phi}(\xi_{0}), \Delta_{\varphi}^{s}T_{\Phi}(\xi_{0}) \rangle_{\varphi}\leq 0 \text{ (note that }T_{\Phi}(\xi_{0})\in \mathfrak{D}(\Delta_{\varphi}^{s})).
\end{align}
Let $(e_{\lambda})_{\lambda\geq 0}$, be the spectral resolution of $\Delta_{\varphi}$ in $\mathcal{A}_{\varphi}$. Then, 
$\Delta_{\varphi}^{s}=\int_{0}^{\infty}\lambda^{s}de_{\lambda}$. From Eq. \eqref{Positivity} it follows that
$\int_{0}^{\infty} \lambda^{s} d\mu_{T_{\Phi}(\xi_{0})}=0$, 
where $\mu_{T_{\Phi}(\xi_{0})}$ is the elementary spectral measure of $\Delta_{\varphi}$ associated to 
the vector $T_{\Phi}(\xi_{0})$. So $[0,\infty) \ni \lambda \mapsto \lambda^{s}$ is $0$ almost everywhere with respect to
$\mu_{T_{\Phi}(\xi_{0})}$, which is impossible unless $T_{\Phi}(\xi_{0})=0$, as $\Delta_{\varphi}$ is nonsingular. So by Eq. \eqref{HahnBanachSeperation} it follows that $\langle \xi_{0},\zeta\rangle_{\rho}=0$ for all $\zeta\in \mathfrak{D}(\Delta_{\rho}^{s})$, i.e. $\xi_{0}=0$. This shows that $\mathfrak{D}_{\rho}$ is indeed a core for $\Delta_{\varphi}^{s}T_{\Phi}$.

\noindent\textbf{Step 5}: In the final step, we proceed to show that $\mathfrak{D}_{\rho}$ is a core for $\overline{T_{\Phi}\Delta_{\rho}^{s}}$ as well. Note that since $T_{\Phi}$ is bounded, it follows that $(T_{\Phi}\Delta_{\rho}^{s})^{*}=\Delta_{\rho}^{s}T_{\Phi}^{*}=\Delta_{\rho}^{s}T_{\Phi^{*}}$, where $\Phi^{*}\in CP(M,N,\varphi,\rho)$ is defined in 
Eq. \eqref{Eq: AccardiCecciniAdjoint}. Reversing the roles of $N$ and $M$ and arguing just as above it follows that $(T_{\Phi}\Delta_{\rho}^{s})^{*}=\Delta_{\rho}^{s}T_{\Phi^{*}}$ is closed with core $\mathfrak{D}_{\varphi}$. It follows that $\overline{T_{\Phi}\Delta_{\rho}^{s}}=(T_{\Phi}\Delta_{\rho}^{s})^{**}=(\Delta_{\rho}^{s}T_{\Phi^{*}})^{*}$. Consequently $(\Delta_{\rho}^{s}T_{\Phi^{*}})^{*}$ is closed.  

We intend to show that $\mathfrak{D}_{\rho}$ is a core for $(\Delta_{\rho}^{s}T_{\Phi^{*}})^{*}$. Once this is established, from Eq. \eqref{AgreeonD_U} it follows that $\overline{T_{\Phi}\Delta_{\rho}^{s}}=\Delta_{\varphi}^{s}T_{\Phi}$.

Let $\zeta^{\prime}\in \mathfrak{D}_{\rho}$ and $\eta\in \mathfrak{D}(\Delta_{\rho}^{s}T_{\Phi^{*}})$. Since $\zeta^{\prime}\in \mathfrak{D}(\Delta_{\rho}^{s})$ and $\Delta_{\rho}^{s}$ is self--adjoint, 
\begin{align*}
\abs{\langle \Delta_{\rho}^{s}T_{\Phi^{*}}(\eta), \zeta^{\prime}\rangle_{\rho}}=\abs{\langle T_{\Phi^{*}}(\eta), \Delta_{\rho}^{s}(\zeta^{\prime})\rangle_{\rho}}\leq \norm{\eta}_{\varphi}\norm{\Delta_{\rho}^{s}(\zeta^{\prime})}_{\rho},
\end{align*}
as $\norm{T_{\Phi^{*}}}=1$. It follows that $\zeta^{\prime} \in \mathfrak{D}((\Delta_{\rho}^{s}T_{\Phi^{*}})^{*})$, i.e. $\mathfrak{D}_{\rho}\subseteq \mathfrak{D}((\Delta_{\rho}^{s}T_{\Phi^{*}})^{*})$. 

Suppose there exists $\xi^{\prime}\in \mathfrak{D}((\Delta_{\rho}^{s}T_{\Phi^{*}})^{*})$ such that for all $\zeta^{\prime}\in\mathfrak{D}_{\rho}$, 
\begin{align}\label{FlipClosed}
\langle \xi^{\prime}, \zeta^{\prime}\rangle_{\rho}&= - \langle (\Delta_{\rho}^{s}T_{\Phi^{*}})^{*}(\xi^{\prime}), (\Delta_{\rho}^{s}T_{\Phi^{*}})^{*}(\zeta^{\prime})\rangle_{\varphi}\\
\nonumber&=- \langle (\Delta_{\rho}^{s}T_{\Phi^{*}})^{*}(\xi^{\prime}), \overline{T_{\Phi}\Delta_{\rho}^{s}}(\zeta^{\prime})\rangle_{\varphi}\\
\nonumber&=- \langle (\Delta_{\rho}^{s}T_{\Phi^{*}})^{*}(\xi^{\prime}), T_{\Phi}\Delta_{\rho}^{s}(\zeta^{\prime})\rangle_{\varphi} \text{ (as } \zeta^{\prime}\in \mathfrak{D}_{\rho}\subseteq \mathfrak{D}(T_{\Phi}\Delta_{\rho}^{s}))\\
\nonumber&=- \langle T_{\Phi^{*}}(\Delta_{\rho}^{s}T_{\Phi^{*}})^{*}(\xi^{\prime}), \Delta_{\rho}^{s}(\zeta^{\prime})\rangle_{\rho}.
\end{align}

Then,  $(\xi^{\prime}, T_{\Phi^{*}}(\Delta_{\rho}^{s}T_{\Phi^{*}})^{*}(\xi^{\prime}))$ is orthogonal to $\mathcal{G}({\Delta_{\rho}^{s}\lvert}_{\mathfrak{D}_{\rho}})$ and because $\mathfrak{D}_{\rho}$ is a core for the closed operator $\Delta_{\rho}^{s}$, by Eq. \eqref{FlipClosed} we obtain 
\begin{align*}
\langle \xi^{\prime}, \zeta^{\prime}\rangle_{\rho}&= - \langle T_{\Phi^{*}}(\Delta_{\rho}^{s}T_{\Phi^{*}})^{*}(\xi^{\prime}), \Delta_{\rho}^{s}(\zeta^{\prime})\rangle_{\rho} \text{ for all }\zeta^{\prime}\in \mathfrak{D}(\Delta_{\rho}^{s}).
\end{align*}
It follows that $T_{\Phi^{*}}(\Delta_{\rho}^{s}T_{\Phi^{*}})^{*}(\xi^{\prime})\in \mathfrak{D}(\Delta_{\rho}^{s})$ $($as $\Delta_{\rho}^{s}=(\Delta_{\rho}^{s})^{*})$ and 
\begin{align*}
\langle \xi^{\prime}, \zeta^{\prime}\rangle_{\rho}&= - \langle (\Delta_{\rho}^{s}T_{\Phi^{*}})(\Delta_{\rho}^{s}T_{\Phi^{*}})^{*}(\xi^{\prime}), \zeta^{\prime}\rangle_{\rho}.
\end{align*}

Because $\mathfrak{D}(\Delta_{\rho}^{s})$ is dense, it follows that $ \Delta_{\rho}^{s}T_{\Phi^{*}}(\Delta_{\rho}^{s}T_{\Phi^{*}})^{*}(\xi^{\prime}) =-\xi^{\prime}$.
Thus, $-\langle \xi^{\prime},\xi^{\prime}\rangle_{\rho}=\langle \Delta_{\rho}^{s}T_{\Phi^{*}}(\Delta_{\rho}^{s}T_{\Phi^{*}})^{*}(\xi^{\prime}),\xi^{\prime}\rangle_{\rho}=\langle (\Delta_{\rho}^{s}T_{\Phi^{*}})^{*}(\xi^{\prime}),(\Delta_{\rho}^{s}T_{\Phi^{*}})^{*}(\xi^{\prime})\rangle_{\rho}\geq 0$. It follows that $(\Delta_{\rho}^{s}T_{\Phi^{*}})^{*}(\xi^{\prime})=0$, which forces $\xi^{\prime}=0$. This Hahn--Banach separation argument shows that $\mathfrak{D}_{\rho}$ is a core for $(\Delta_{\rho}^{s}T_{\Phi^{*}})^{*}$, and the proof is complete.
\end{proof}

\begin{remark}
\emph{It is not possible to obtain that $T_{\Phi}\Delta_{\rho}^{z}=\Delta_{\varphi}^{z}T_{\Phi}$ in general in Theorem \ref{Commute}. To see this, note that if $\Phi(\cdot)=\rho(\cdot)1_{M}$, then $\Delta_{\varphi}^{z}T_{\Phi}$ is an everywhere defined bounded operator, while $T_{\Phi}\Delta_{\rho}^{z}$ is only densely defined.}
\end{remark}

\begin{theorem}\label{CommuteJ}
Let $\Phi\in CP(N,M,\rho,\varphi)$. Then
\begin{align*}
&(i) \text{ }\overline{\Delta_{\varphi}^{-s}T_{\Phi}\Delta_{\rho}^{s}}=T_{\Phi} \text{ for all }s\in \mathbb{R};\\
&(ii) \text{ }J_{\varphi} T_{\Phi} J_{\rho}=T_{\Phi};\\
&(iii) \text{ }\overline{S_{\varphi}T_{\Phi}S_{\rho}}=T_{\Phi}.
\end{align*}

\end{theorem}

\begin{proof}
From Theorem \ref{Commute} we have $\overline{T_{\Phi}\Delta_{\rho}^{s}}=\Delta_{\varphi}^{s}T_{\Phi}$ for all $s\in\mathbb{R}$. Thus, $\Delta_{\varphi}^{-s}T_{\Phi}\Delta_{\rho}^{s}\subseteq T_{\Phi}$. But $\Delta_{\varphi}^{-s}T_{\Phi}\Delta_{\rho}^{s}$ is densely defined. To see this, note that if $\xi\in \mathfrak{D}_{\rho}$ $($where $\mathfrak{D}_{\rho}$ is defined in the statement of Theorem \ref{Commute}$)$, then $T_{\Phi}(\xi)\in \mathfrak{D}_{\varphi}$ $($where $\mathfrak{D}_{\varphi}$ is defined before Eq. \eqref{AgreeonD_U}$)$. 
By Theorem \ref{Commute} it follows that $T_{\Phi}(\Delta_{\rho}^{s}\xi)= \Delta_{\varphi}^{s}T_{\Phi}(\xi)$. Thus $T_{\Phi}(\Delta_{\rho}^{s}\xi)\in \mathfrak{D}(\Delta_{\varphi}^{-s})$, since $\Delta_{\varphi}^{-s}=(\Delta_{\varphi}^{s})^{-1}$. So $\xi\in \mathfrak{D}(\Delta_{\varphi}^{-s}T_{\Phi}\Delta_{\rho}^{s})$. It follows that 
\begin{align*}
\overline{\Delta_{\varphi}^{-s}T_{\Phi}\Delta_{\rho}^{s}}=T_{\Phi}.
\end{align*}
When $s=-\frac{1}{2}$, we have $\overline{\Delta_{\varphi}^{\frac{1}{2}}T_{\Phi}\Delta_{\rho}^{-\frac{1}{2}}}=T_{\Phi}$. But then 
\begin{align*}
S_{\varphi}T_{\Phi}S_{\rho} = J_{\varphi}\Delta_{\varphi}^{\frac{1}{2}} T_{\Phi} \Delta_{\rho}^{-\frac{1}{2}}J_{\rho} \subseteq  J_{\varphi} T_{\Phi} J_{\rho}.
\end{align*}
Since $N\Omega_{\rho}$ is a core for $S_{\rho}$, the operator $S_{\varphi}T_{\Phi}S_{\rho}$ is densely defined and consequently 
\begin{align*}
\overline{S_{\varphi}T_{\Phi}S_{\rho}} =  J_{\varphi} T_{\Phi} J_{\rho}.
\end{align*}
Again, since $S_{0,\varphi}T_{\Phi}S_{0,\rho}=T_{\Phi}$ on $N\Omega_{\rho}$, we obtain $\overline{S_{0,\varphi}T_{\Phi}S_{0,\rho}}=T_{\Phi}$. But $S_{0,\varphi}T_{\Phi}S_{0,\rho}\subseteq S_{\varphi}T_{\Phi}S_{\rho}$, and so $T_{\Phi}\subseteq J_{\varphi} T_{\Phi} J_{\rho}$. Finally, because $T_{\Phi}$ and $J_{\varphi} T_{\Phi} J_{\rho}$ are bounded we have equality, i.e., $J_{\varphi} T_{\Phi} J_{\rho}=T_{\Phi}$. 
\end{proof}

%%%%%%%%%%%%%%%%%%%%%%%%%%%%%%%%%%%%%%%%%%%%%%%%%%%%%%%%%%%%%%%%%%%%%%%%%%%%%%%%%%%%%%%%%%%%%%%%%

%%%%%%%%%%%%%%%%%%%%%%%%%%%%%%%%%%%%%%%%%%%%%%%%%%%%%%%%%%%%%%%%%%%%%%%%%%%%%%%%%%%%%%%%%%%%%%%%%%%%%%%%%%

\end{document}